\documentclass[a4paper,12pt,reqno]{amsart}
\usepackage{latexsym, mathabx}
\usepackage{amsmath,amsthm,amssymb,amscd}
\usepackage{fullpage}
\usepackage{xcolor}
\usepackage{parskip}
\usepackage{mathtools}
\usepackage{thmtools}
\mathtoolsset{showonlyrefs}
\usepackage{csquotes}
\usepackage{comment}
\usepackage{pdfpages}
\usepackage[activate={true,nocompatibility},final,tracking=true,kerning=true,spacing=true]{microtype}
\usepackage[pdfencoding=unicode,pdftex,bookmarks=true,bookmarksnumbered]{hyperref}
\definecolor{citecolour}{rgb}{0.0, 0.0, 0.8}
\colorlet{linkcolour}{green!50!black}
\hypersetup{colorlinks,breaklinks,
		linkcolor=linkcolour,citecolor=citecolour,
		filecolor=linkcolour, urlcolor=linkcolour}
\definecolor{paperblue}{RGB}{20,60,130}
\definecolor{paperpurple}{RGB}{95,55,135}
\usepackage{txfonts,pxfonts,tikz} 
\usepackage{tgschola}
\usepackage[T1]{fontenc}

\newtheorem{prevtheorem}{Theorem}

\newtheorem{theorem}{Theorem}

\newtheorem{proposition}{Proposition}

\newtheorem{lemma}[theorem]{Lemma}
\newtheorem{corollary}{Corollary}

\newtheorem{question}{Question}
\newtheorem{definition}{Definition}

\newtheorem*{corollary*}{Corollary}

\newenvironment{proofofA}{{\it {Proof of Theorem \ref{Thm:main}.} }}{\hfill $\blacksquare$ \\}

\def\Ker{\mathrm{Ker}}
\def\Ima{\mathrm{Im}}
\def\Z{\mathbf{Z}}

\def\R{\mathbf{R}}
\def\F{\mathbf{F}}

\def\Aut{\mathrm{Aut}}

\def\End{\mathrm{End}}
\def\GL{\mathrm{GL}}
\def\Mat{\mathrm{Mat}}
\def\P{\mathcal{P}}

\title{A simultaneous variant of the Hall-Paige theorem and applications}
\author{Maria Loukaki }
\address{Department of Mathematics and Applied Mathematics, University of Crete, Voutes Campus, 70013 Heraklion, Greece}
\email{mloukaki@uoc.gr}

\date{}

\begin{document}

\begin{abstract}
    For $k\geq 1$ , we say that a finite group $G$  satisfies property $\mathcal P_k$ if, for every $\beta_1,\ldots,\beta_k\in\operatorname{Aut}(G)$, there exists a bijection $f$ of $G$  such that each map $g\mapsto \beta_i(g)f(g)$ is a bijection. For $k=1$ this is equivalent to the classical Hall--Paige complete mapping property. We prove that if $p$ is the smallest prime divisor of $|G|$, then $G$ satisfies $\mathcal P_k$ for every $k<p-1$ and we show that this bound is sharp. For finite abelian $p$-groups, the bijection $f$ may moreover be chosen to be an automorphism, and for noncyclic homocyclic $p$-groups this remains true for $k=p-1$. Finally, we interpret $\mathcal P_k$ in terms of common transversals and common orthogonal mates of Latin squares generated from automorphisms and  the Cayley table  of a finite group $G$.
\end{abstract}


\maketitle
\section{Introduction}
A bijection $\theta: G \to G$  of a finite group $G$ is a {\it complete mapping} of $G$ if the mapping $x \to  x\theta(x)$ is a bijection of $G$. 
Complete mappings were introduced by Mann in 1942 in connection with the construction of orthogonal Latin squares \cite{Man42}, and were subsequently studied systematically by Paige \cite{Pai51}. They arise naturally from the Cayley table of a group: a finite group admits a complete mapping if and only if its Cayley table has a transversal, or equivalently an orthogonal mate.

Hall and Paige proved in 1955 that a necessary condition for a finite group $G$ to admit a complete mapping is that the Sylow $2$-subgroup of $G$ is either trivial or noncyclic, and conjectured that this condition is also sufficient \cite{HP55}. They proved the conjecture for several important classes of groups, including all solvable groups. The conjecture remained open for more than fifty years and was eventually completed through the work of Wilcox, Evans and Bray \cite{Wil09,Eva09}. Thus a finite group admits a complete mapping if and only if its Sylow $2$-subgroup is trivial or noncyclic. Recently, complete mappings and the Hall--Paige theorem have been revisited from different points of  view;  the work of Eberhard, Manners and Mrazovi'c  obtained asymptotic formulae for the number of complete mappings of large groups in \cite{EMM22}, 
while  M{\"u}yesser and Pokrovskiy use techniques from probabilistic combinatorics to prove  a "random version" of the Hall--Paige theorem \cite{MP25}.

 A related notion is that of an {\it orthomorphism}. A bijection $\theta: G \to G$ of a finite group $G$ is called an orthomorphism if the mapping $x \to x^{-1} \theta(x)$ is also a bijection, and $\theta$ is called a   {\it strong complete mapping} if both $x\mapsto x\theta(x)$ and $x\mapsto x^{-1}\theta(x)$ are bijections. Thus a strong complete mapping is simultaneously a complete mapping and an orthomorphism.
 Evans raised the problem of characterising the groups 
 which contain strong complete mappings (see \cite{Eva12}). In particular, Evans proved that a finite abelian group admits a strong complete mapping if and only if neither its Sylow $2$-subgroup nor its Sylow $3$-subgroup is nontrivial and cyclic \cite{Eva12}. Strong complete mappings have also appeared in connection with orthogonal Latin squares and other combinatorial constructions.

The purpose of the present paper is to consider a simultaneous version of complete mappings.  Instead of a single bijection being both a complete mapping and an orthomorphism (as in the case of strong complete maps)  we require one bijection to form a complete-mapping-type pair with each automorphism in  an arbitrary collection of automorphisms. It is worth mentioning that the definition of property $\P_k$ arose from our attempt to generalize the findings in \cite{LS} to over three subgroups,  specifically,  by seeking common transversals for four or more subgroups of the same index within an abelian group (see Section \ref{sec:applications}).(A transversal of a subgroup $H$ in a group $G$ is a set with one representative from each coset of $H$ in $G$.)

We  define the property $\P_k$ for any integer $k\geq 1$ as follows.
\begin{definition}
We say that a group $G$ satisfies $\P_k$, for some $k \geq 1$, if
for every $\beta_1,\dots,\beta_k \in \Aut(G)$, there exists a bijection  $f$ of $G$ such that every map $g \mapsto \beta_i(g)f(g)$  is a bijection  of $G$. 
\end{definition}

The case $k=1$ is exactly the classical complete-mapping problem. Indeed, if a complete map $f$  exists in $G$  then for any $\beta_1 \in \Aut(G)$, the map $ f \circ \beta_1$ is a bijection of $G$ while  $g \to  \beta_1(g)f(\beta_1(g)) $ is also a bijection. Hence the Hall--Paige theorem gives a complete characterization of the groups satisfying $\P_1$.

Things get more interesting when $k=2$. While every abelian group satisfying property $\P_2$ is straightforwardly shown to possess a strong complete map, these two properties are not equivalent even for abelian groups. For instance, consider the abelian group $G = \Z_2 \times \Z_2$. According to Evans' criterion for abelian groups, $G$ admits a strong complete mapping. However, property $\P_2$ fails for the pair of automorphisms $\beta_1, id$, where $\beta_1$ interchanges the two copies of $\Z_2$, i.e., $\beta_1(1,0)=(0,1)$ and $\beta_1(0,1)=(1,0)$. (This can be easily verified as follows; first  translate $f$  so that $f((0,0))=(0,0)$. If both $f$  and $id +f$  are permutations on the 3 non-zero elements of $G$, then $f$  acts on these as a fixed-point-free permutation, hence as one of the two $3$-cycles. But then one of those cycles gives $(\beta_1+f)((1,0))=(0,0)$, and the other gives $(\beta_1+f)((0,1))=(0,0)$, while $(\beta_1+f)((0,0))=(0,0)$.) 

Our main existence result reads:
\begin{prevtheorem}\label{Thm:main}
Assume $G$ is a finite group. Let $p$ be the smallest prime divisor of the order $|G|$ of $G$. Then $G$ satisfies $\P_k$ for all $1 \leq k <p-1$. 
\end{prevtheorem}

As an application we prove the following (see Section~4 for all  the relevant definitions):
 \begin{corollary} \label{col:commonTrans+Mate}
     Let $G$ be a finite group whose smallest prime divisor is $p$. If $k < p -1$, then for any automorphisms $\beta_1, \dots, \beta_k \in \Aut(G)$  the Latin squares $L_i(x,y)=\beta_i(x)y$ have a common transversal and a common orthogonal mate.
\end{corollary}

\autoref{Thm:main} offers insights beyond the Hall-Paige theorem only for odd-order groups (and thus solvable), where $k \geq 2$. For its proof we first proved  it for finite abelian $p$-groups and then we extend it to arbitrary finite solvable groups. In fact,  for abelian $p$-groups, we can more strongly conclude that given $\beta_1,\ldots,\beta_k \in \Aut(G)$, the bijection $f$ can be an automorphism itself, making all maps $\beta_i+f$ automorphisms as well (see \autoref{Thm:k<p-1-abelian-p-case}).

The threshold $p-1$  cannot be improved in general;  we prove that  cyclic $p$-groups do not satisfy  $\P_{p-1}$ for any prime $p$ (see \autoref{thm:Cyclic-not-P_{p-1}} ). 
On the other hand, we show that the obstruction at $p-1$ disappears for non-cyclic homocyclic $p$-groups: every such group satisfies $\mathcal P_{p-1}$, with the simultaneous bijection chosen (again)  to be an automorphism (see \autoref{Prop.Homocyclic.k=p-1}). These results suggest the following natural analogue, for abelian $p$-groups, of the Hall--Paige characterization: does every noncyclic finite abelian $p$-group satisfy $\P_{p-1}$? We have not managed to answer this question. 

The paper is organized as follows. In Section~2 we prove an extension result that groups with  property  $\P_k$ satisfy and prove \autoref{Thm:main} along with special results for abelian groups and homocyclic $p$-groups. In  Section~3  we prove that the upper bound $k < p-1$ is sharp and finally in   Section~4 we apply the theory to Latin squares and  transversals of subgroups.

\section{simultaneously complete mappings}
This section aims to prove \autoref{Thm:main}. We begin by noting some  basic properties of $\P_k$.  First, if a group $G$ satisfies $\P_k$, it also satisfies $\P_r$ for all $1 \leq r \leq k$. 
Second, an extension result, a simultaneous analogue of Hall and Paige’s extension theorem for complete mappings, as described in \cite[Theorem 1 and Corollary 2]{HP55}, also holds for characteristic subgroups, as the next proposition shows.
\begin{proposition}\label{prop: Extension-property}
    Assume $N $ is a characteristic subgroup of the finite group $G$. If $N$ and $G/N$ satisfy property $\P_k$ for some integer $k$ then $G$ satisfies $\P_k$. 
\end{proposition}

\begin{proof}
We clearly can assume that $N < G$. Let $\{\beta_i\}_{i \leq k}$ be automorphisms of $G$. 
 Since $N$ is characteristic in $G$ and each $\beta_i \in \Aut(G)$, each $\beta_i$ induces an automorphism $\bar{\beta}_i \in \Aut(G/N)$ defined by 
\[
\bar{\beta}_i(gN) := \beta_i(g) N.
\]

Because $G/N$ satisfies $\P_k$,  there exists a bijection $\bar{\theta} : G/N \to G/N$ such that 
\[
gN \to \bar{\beta}_i(gN)\bar{\theta}(gN) 
\]
 is a bijection of $G/N$ for all $i=1, \dots , k$.

Let $T$ be a transversal of $N$ in $G$. Since $\bar{\theta}$ is a bijection of $G/N$, there exists a permutation $\sigma$ of $T$  such that 
\[
\bar{\theta}(t N) = \sigma(t) N,
\]
for all $t \in T$. We also note  that  every element $g \in G$ is written uniquely as $g= tn$ with  $t \in T$ and $n \in N$.

Now fix $t \in T$ and observe that for every $i =1, \dots, k$, the map $\alpha_{i,t}: G \to G$ defined as 
\[
\alpha_{i,t}(g)= \sigma(t)^{-1} \beta_i(g) \sigma(t),
\]
is also an automorphism of $G$ because $\beta_i$ is. Furthermore,  
each restriction $ \alpha_{i,t} |_N $ lies in $\Aut(N)$, because  $N$ is characteristic subgroup in $G$. Now $N$ also satisfies $\P_k$ hence  we obtain  a bijection  $\phi_t $  such  that, for each $i=1, \dots, k$,  the map
\[
N \ni n \to \alpha_{i,t}(n)\phi_t(n)
\]
is a bijection  of $N$.

We define $f: G \to G$ as 
\[
f(tn):= \sigma(t) \phi_t(n)
\]
for every $t \in T$ and $n \in N$ and we claim that $f$ is a bijection of $G$. Indeed, assume 
\begin{equation}\label{eq:f-bijection}
\sigma(t) \phi_t(n)= \sigma(t') \phi_{t'}(n')
\end{equation}
where $(t, n) \neq (t',n')$. Since $T$ is a transversal of $N$ and $\phi_t ( n), \phi_{t'}(n') \in N$, we must have $\sigma(t) = \sigma(t')$, which implies $t=t'$. Equation   \eqref{eq:f-bijection} then simplifies to $\phi_t(n)= \phi_{t}(n')$. As $\phi_t$ is a bijection, it follows that $n=n'$, contradicting our initial assumption and thus proving our  claim.

To complete the proof, it remains to show that $\beta_i(tn) f(tn)$ is a bijection of $G$,  for every $i=1, \dots, k$. Assume $g\neq g'$ are  distinct elements of $G$, with $g =tn$  and $g'=t'n'$, so  $(t,n) \neq (t',n')$. If 
\begin{equation}\label{eq:b_if-bijection}
    \beta_i(tn)\sigma(t)\phi_t(n)= \beta_i(t'n')\sigma(t')\phi_{t'}(n') 
\end{equation}
then      $\bar{\beta}_i(tN) \bar{\theta}(tN)= \bar{\beta}_i(t'N) \bar{\theta}(t'N)$. 
Hence $tN= t'N$ as $gN \to \bar{\beta}_i(gN)\bar{\theta}(gN) $
 is a bijection of $G/N$ and so $t=t'$, as $T$ is a transversal.  Thus  \eqref{eq:b_if-bijection} implies that 
\[
\sigma(t)^{-1}\beta_{i}(n)\sigma (t) \phi_t(n)= \sigma(t)^{-1}\beta_i(n') \sigma(t) \phi_t(n'),
\] 
since $\beta_i$ is an automorphism of $G$.  So   $ \alpha_{i,t}(n)\phi_t(n)= \alpha_{i,t}(n')\phi_t(n')$ and thus $n=n'$, as the map $n \to \alpha_{i,t}(n)\phi_t(n)$ is a bijection. This leads to $(t,n)= (t',n')$, a contradiction that completes the proof of the proposition.   
\end{proof}

Now we turn our attention to abelian $p$-groups,  for some prime $p$. 
 It is well known that the automorphism group of a finite abelian group $G$ can be  described using the classification of $G$,  see  \cite{Hillar2007}. We apply their results \cite[Theorem 3.3 and 3.6]{Hillar2007} to  characterize  the  automorphisms of abelian $p$-groups. Let 
\[
G \cong \Z_{p^{e_1}} \times \cdots \times \Z_{p^{e_n}}
\]
be an abelian $p$-group with $1 \leq e_1 \leq \cdots \leq e_n$. In this form, an element of $G$ is a vector $(\overline{h}_1,\dots,\overline{h}_n)$, where each $h_i \in \Z$ and $\overline{h}_i \in \Z_{p^{e_i}}$ is the residue class of $h_i$ modulo $p^{e_i}$. If we denote by 
\[
\pi_i : \Z \to \Z_{p^{e_i}}
\]
the natural reduction map modulo $p^{e_i}$, then $\overline{h}_i = \pi_i(h_i)$. Define 
\[
\pi : \Z^n \to G
\]
to be the homomorphism
\[
\pi(h_1,\dots,h_n) = (\pi_1(h_1),\dots,\pi_n(h_n)) = (\overline{h}_1,\dots,\overline{h}_n).
\]
We also define $R_G$ to be the matrix ring (see  \cite[Lemma 3.2]{Hillar2007} )
\[
R_G = \{ A = (a_{ij}) \in \Mat_n(\Z) \mid  \, \,  p^{e_i-e_j } \mid a_{ij} \, \, \text { for } \, 1 \leq j \leq i \leq n \}.
\]
With this setup, the following statement holds:
\begin{theorem}[Theorem 3.3 and   3.6 in  \cite{Hillar2007}]\label{Thm:AutAbelian}
    Let  $G \cong \Z_{p^{e_1}} \times \cdots \times \Z_{p^{e_n}}$ be an abelian $p$-group, with $1 \leq e_1 \leq  \cdots \leq e_n$. Then the map $\psi: R_G \to \End(G)$  defined as  
    \[
\psi(A) (\overline{h}_1,\dots,\overline{h}_n)^T= \pi\big(A(h_1, \cdots, h_n)^T\big)
    \]
    is a surjective ring  homomorphism. Furthermore an endomorphism $\beta= \psi(A)$ of $G$ is an automorphism if and only if the reduction of  $ A$ modulo $p $  is invertible in $GL_n(\F_p)$.   
\end{theorem}

We can now prove 
\begin{theorem}\label{Thm:k<p-1-abelian-p-case}
  Assume $G$ is a finite abelian $p$-group and $k < p-1$ an integer. If $\beta_1, \dots, \beta_k$ are automorphisms of $G$, then there exists an automorphism $f$ of $G$ such that for each $i \in \{1, \dots, k\}$, the map $g \mapsto \beta_i(g)+f(g)$ is an automorphism of $G$. Consequently, $G$ satisfies property $\P_k$ for all $k < p-1$. 
\end{theorem}
                                                    
\begin{proof}

First assume that $G \cong \mathbb Z_{p^e}^n$ is a homocyclic $p$-group. By \autoref{Thm:AutAbelian}, for each automorphism $\beta_i$ there exists a matrix $A_i \in R_G$, where $R_G=\Mat_n(\Z)$ in the homocyclic case, such that $\beta_i=\psi(A_i)$ and $A_i \bmod p$ is non-singular, for all $i=1,\dots,k$. We observe that it is enough to find a matrix $B\in \Mat_n(\Z))$ such that both $B\bmod p$ and $(A_i +B )\bmod p$ are invertible for every $i$, as then $f:=\psi(B)$ will be the required automorphism of $G$.

Define
\begin{equation*}
T=\{X\in \Mat_n(\F_p)\mid \det X=0\},
\tag{1}
\end{equation*}
and denote by $\overline{ A_i}$ the reduction of $A_i$ modulo $p$. The condition that $A_i+B$ is invertible modulo $p$ is equivalent to requiring that
\begin{equation*}
\overline{ B} \notin T-\overline{ A_i}
\end{equation*}
for each $i=1,\dots,k$. (By $T-\overline{ A_i}$ we denote the set  consisting of the matrices $ X-\overline{ A_i }$ with  $X \in T $.) 

It is enough to prove that
\begin{equation*}
\big|\GL_n(\F_p)\big| - \big| \bigcup_{i=1}^k (T-\overline{A_i}) \big|>0 .
\end{equation*}

We have
\[
|\GL_n(\mathbb F_p)|= 
p^{n^2}\prod_{t=1}^n(1-p^{-t}),
\]
and, for each $i$,
\[
|T-\overline {A_i}|=|T|= p^{n^2}-|\GL_n(\F_p)|.
\]
Hence
\[
\left|\bigcup_{i=1}^k (T-\overline {A_i})\right|
\leq
k p^{n^2}\left[1-\prod_{t=1}^n(1-p^{-t})\right].
\]
Thus
\begin{align}\label{eq:countAutom}
\left|\GL_n(\F_p) \right| - \left| \bigcup_{i=1}^k (T-\overline A_i)\right|  &\geq 
p^{n^2}  \left[  \prod_{t=1}^n(1-p^{-t}) -  k\left(1-\prod_{t=1}^n(1-p^{-t})\right)  \right] \\
&=  p^{n^2}  \left[ (k+1)\prod_{t=1}^n(1-p^{-t})-k \right].
\end{align}

Furthermore,
\begin{align*}
\prod_{t=1}^n(1-p^{-t})
&\geq 1-\sum_{t=1}^n \frac{1}{p^t}  \\
&= 1-\frac{p^{-1}(1-p^{-n})}{1-p^{-1}} \\
&= 1-\frac{1-p^{-n}}{p-1} \\
&=\frac{p-2+p^{-n}}{p-1} > \frac{p-2}{p-1}.
\end{align*}
Since $k\leq p-2$, and the function $x\mapsto x/(x+1)$  is increasing, it follows that
$\frac{p-2}{p-1}\geq \frac{k}{k+1}$. Therefore
$ \prod_{t=1}^n(1-p^{-t})>\frac{k}{k+1}$. Together with \eqref{eq:countAutom}, this gives
\[
\left|\GL_n(\F_p)\right| - \left| \bigcup_{i=1}^k (T-\overline{A_i})\right|>0.
\]
Hence the theorem holds in this homocyclic case.

Now suppose that
 \[
  G \cong (\Z_{p^{e_1}})^{m_1}  \times \cdots \times (\Z_{p^{e_r}})^{m_r}
  \]
where  $1 \leq e_1 < \cdots < e_r$ and $n=m_1+\dots+m_r$.

As earlier, for each automorphism $\beta_i \in \Aut(G)$ there exists a matrix $A_i \in \R_G$ such that $\beta_i = \psi(A_i)$ and $A_i \bmod p$ is nonsingular,  for all $i=1, \dots, k$. Notice that,  in view of the definition of $R_G$ the reduced matrices  $\overline{A_i}$ are block upper triangular matrices of the form
\[
\overline A_i=
\begin{pmatrix}
\overline A_i^{1} & * & \cdots & * \\
0 & \overline A_i^{2} & \cdots & * \\
\vdots & \vdots & \ddots & \vdots \\
0 & 0 & \cdots & \overline{A}_i^{r}
\end{pmatrix}
\]
where $\overline {A}_i^j \in \GL_{m_j}(\F_p)$, for all $j=1, \dots, r$.

For every fixed $j$, the  argument for the homocyclic case provides an invertible matrix $\overline {B}^j \in \GL_{m_j}(\F_p)$ such that the matrices  $\overline{A}_i^j + \overline{B}^j  $ are all invertible and thus elements in $\GL_{m_j}(\F_p)$.

Now define the block diagonal matrix
\[
\overline B:=
\begin{pmatrix}
\overline B^{1} & 0 & \cdots & 0 \\
0 & \overline B^{2} & \cdots & 0 \\
\vdots & \vdots & \ddots & \vdots \\
0 & 0 & \cdots & \overline B^{r}
\end{pmatrix}.
\]
Clearly $ \overline B\in \GL_n(\F_p)$, because all its diagonal blocks are invertible.

Moreover, for every $i=1,\dots,k$, the matrix $\overline A_i + \overline B $ is again block upper triangular, and its diagonal blocks are precisely $\overline A_i^{j} + \overline B^{j} $,  for  $ j=1,\dots,r$.
These blocks are invertible by construction. Hence
\[
\overline A_i + \overline B\,  \in \GL_n(\F_p)  \, \, 
\text{for all } i=1,\dots,k.
\]

Choose a matrix $B\in R_G$ whose reduction modulo $p$ is $\overline B$. This is possible since $\overline B$ is block diagonal, so it is compatible with the divisibility conditions defining $R_G$. By \autoref{Thm:AutAbelian}, the map $f:=\psi(B) $ is an automorphism of $G$. 
Furthermore, for every $i=1,\dots,k$, the reduction of $A_i + B $ modulo $p$ is
$\overline A_i + \overline B$  which is non-singular.  Therefore, again by \autoref{Thm:AutAbelian}, the map 
$\beta_i +f =\psi(A_i + B )$  is an automorphism of $G$, for every $i=1,\dots,k$.

Hence $f$ has the desired properties, and the proof of the theorem is complete.
\end{proof}

For homocyclic that are  not cyclic $p$-groups, we can extend Theorem \ref{Thm:k<p-1-abelian-p-case} to $k=p-1$, as detailed in our next result.

\begin{proposition}\label{Prop.Homocyclic.k=p-1}
    Assume $G$ is a noncyclic homocyclic  $p$-group. If $\beta_1, \dots, \beta_{p-1}$ are automorphisms of $G$, then there exists an automorphism $f$ of $G$ such that for each $i \in \{1, \dots, {p-1}\}$, the map $g \mapsto \beta_i(g)+f(g)$ is an automorphism  of $G$. Consequently, $G$ satisfies property $\P_{p-1}$.   
\end{proposition}

\begin{proof}
 Let $G = \Z_{p^e}^n$ for $n \geq 2$. Following the notation from the proof of Theorem \ref{Thm:k<p-1-abelian-p-case}, for each $\beta_i$ (where $i \leq p-1$), there is a corresponding matrix $A_i \in \Mat_n(\Z)$. We need to find a matrix $B \in \Mat_n(\Z)$ such that $\overline{B}$ and all $\overline{A_i} + \overline{B}$ are invertible in $\Mat_n(\F_p)$, for $i=1, \ldots, p-1$. Equivalently it is enough to show  that 
 \[
\GL_n(\F_p) \neq \bigcup_{i=1}^{p-1} \big[(T -\overline{A_i}) \cap \GL_n(\F_p)\big], 
 \]
 where $T = \{X \in \Mat_n(\F_p) \mid \det X = 0\}$. For $i=1, \dots, p-1$,  let 
 \[
K_i= (T-\overline{A_i} ) \cap \GL_n(\F_p),
 \]
 and $K_0= (T+ I_n) \cap \GL_n(\F_p)$. (We could have $K_0$ being one of the $K_i$ for some $i$.) 
 So to prove the proposition it suffices to show 
\begin{equation}\label{eq.sum-invertible-eigenvalue1}
    \big|\GL_n(\F_p)\big| > \big|\bigcup_{i=1}^{p-1} K_i\big|.
\end{equation}

We observe that $|K_i|= |K_0|$ for every $i$, as  the map $S \to -S (\overline{A_i})^{-1}$ provides  a bijection between $K_i$ and $K_0$. 

Let us note that 
\begin{align*}
K_0 &= \{ X+I_n=S \in GL_n(\F_p) \mid X \text{ singular } \} \\
&= \{ S \in \GL_n(\F_p) \mid \det{(S-I_n)}=0 \}\\
&= \{ S \in \GL_n(\F_p) \mid S v= v \, \, \text{ for some } \, 0 \neq v \in \F_{p}^n \}, 
\end{align*}
i.e. $K_0$ consists of those invertible matrices $S \in \GL_n(\F_p)$ whose $1$ is an eigenvalue.

To compute the cardinality of $K_0$ we look at the set 
\[
R= \{ (S, v) \in \GL_n(\F_p) \times \F_p^{n, *} \mid  Sv = v \}, 
\]
where $ \F_p^{n, *} $ are the non-zero vectors of $\F_p^n$.
Note that $\GL_n(\F_p)$ acts transitively on $\F_p^{n, *}$ hence the Orbit-Stabilizer theorem implies that for any fixed $v \in \F_p^{n,*}$, the stabilizer $\mathrm{Stab}_v$  in $\GL_n(\F_p)$ has order equal to  $\frac{|\GL_n(\F_p)|}{ |\F_p^{n,*}|}$,  as the orbit of $v$ is $\F_p^{n, *}$. 
Hence 
\[
|\mathrm{Stab}_v|= \frac{|\GL_n(\F_p)|}{p^n-1}
\]
and thus 
\begin{equation}\label{eq.|R|}
 |R| = \sum_{v \in \F_p^{n, *}}|\mathrm{Stab}_v|=  (p^n-1)|\mathrm{Stab}_v|=|\GL_n(\F_p)|.    
\end{equation}

On the other hand, for any fixed $S \in K_0$ if $E_S$ is the eigenspace of $S$ for the eigenvalue $1$ then $|R|= \sum_{S \in K_0} (|E_S|-1)$. 
But for any $S \in K_0$ its eigenspace for $1$ has dimension at least $1$ and thus $|E_S| \geq p$. One of those matrices of $K_0$   is $I_n$ with $E_{I_n}=\F_p^n$. Hence 
\[
|R|= \sum_{S \in K_0} (|E_S|-1) \geq (|K_0|-1)(p-1)+ p^n-1.
\]
This, along with \eqref{eq.|R|} implies 
\[
|K_0|\leq  \frac{|GL_n(\F_p)|}{p-1} - \frac{p^n-1}{p-1} +1 <  \frac{|GL_n(\F_p)|}{p-1}, 
\]
where the last inequality is strict as $n \geq 2$.
So 
\[
\big| \bigcup_{i=1}^{p-1} K_i\big| \leq \sum_{i=1}^{p-1} |K_i| = (p-1) |K_0| < {|GL_n(\F_p)|}.
\]
This completes the proof of the proposition.
\end{proof}

Using \autoref{prop: Extension-property} and \autoref{Thm:k<p-1-abelian-p-case} we can now prove that every group $G$ satisfies $\P_k$, provided  $k$ is  less  than $p-1$ for every $p $ prime divisor of the order $|G|$ of $G$. Let us point out  that the bijection $f$ in the following theorem is not necessarily an automorphism, unlike in \autoref{Thm:k<p-1-abelian-p-case} and \autoref{Prop.Homocyclic.k=p-1}.

\begin{proofofA}
Since $1 \leq k < p-1$, it follows that $p \geq 3$, so $G$ has odd order and is therefore solvable. Now we proceed by induction on $|G|$. 

Let $N$ be a non-trivial minimal characteristic subgroup of $G$. Because $G$ is solvable $N$ must be  an elementary abelian $q$-group for some prime number $q$. To see this first observe that $N'$ is a characteristic proper (as $N$ is solvable) subgroup of $N$ and thus of $G$. Hence $N'=1$ forcing $N$ being abelian. As the Sylow subgroups of $N$ are characteristic we conclude that $N$ is an abelian $q$-group.  Similarly, the subgroup $\Omega_1(N) = \{ n \in N \mid n^q=1\}$ is characteristic in $N$ and so minimality forces $\Omega_1(N)= N$. Thus  
\[
N \cong \Z_q^n, 
\]
with $q \geq p$, since  $p$ is the smallest prime divisor of $|G|$. Consequently,  $k < p-1\leq q-1$. 

If $N=G$, then $G$ is an elementary abelian $q$-group, and the result follows directly from \autoref{Thm:k<p-1-abelian-p-case}. Hence we may assume that $N<G$.
Consider the quotient group $G/N$, whose smallest prime divisor is at least $p$. 
By induction $\P_k$ holds for $G/N$ and it also holds for $N$. Thus \autoref{prop: Extension-property} implies that $G$ satisfies $\P_k$ completing the induction and the proof of the theorem.
\end{proofofA}

\section{Cyclic $p$-groups do not satisfy $\P_{p-1}$}

 \autoref{Thm:k<p-1-abelian-p-case} shows that abelian $p$-groups satisfy $\P_k$ for any $k < p-1$ (with the stronger conclusion that $f$ can be picked to be an automorphism  of $G$).  We will now prove  that $k < p-1$ is an  optimal uniform bound for abelian $p$-groups,  beginning with two lemmata.

The first  lemma  imitates the well-known property of cyclic $2$-groups, also crucial to the existence of complete maps, where the sum of their elements equals the group's unique non-zero involution.
\begin{lemma}\label{lem:Sum-In-Zpn}
Assume $p$ is a prime number and $n \geq 1$. Then 
\[    \sum_{x=1}^{p^n}  x^{p-1} \equiv -p^{n-1} \pmod{p^n}.
\]
\end{lemma}

\begin{proof}
We first note that for $p=2$ the sum $\sum_{x \in \Z_{2^n}} x$ equals the unique nonzero involution  of $\Z_{2^n} $, that is the element $2^{n-1} \in \Z_{2^n}$ and so the lemma holds for $p=2$.

We may assume now that $p$ is odd and    we proceed by induction on $n$. The base case, $n=1$, is trivial:
\[
\sum_{x =1}^{p} x^{p-1} \equiv \sum_{x=1}^{p-1} 1 = p-1 \equiv -p^0 \pmod{p}.
\]
Assume the hypothesis holds for values less than $n$. We will now prove the inductive step for $n$. To compute the sum, we first split the elements of $\Z_{p^n}$ into two sets,  the set of units of the multiplicative  $\Z_{p^n}^*$ and the multiples of $p$:
\[
    \sum_{x=1}^{p^n} x^{p-1}= \sum_{(x,p)=1 \atop  1\leq x \leq p^n } x^{p-1} + \sum_{k=1}^{p^{n-1}} (kp)^{p-1}.
\]

Applying the inductive hypothesis (for $n-1$)
the second part of the sum simplifies to:
\[
\sum_{k=1}^{p^{n-1}} (kp)^{p-1}= p^{p-1} \sum_{k=1}^{p^{n-1}} k^{p-1} = p^{p-1} ( rp^{n-1} -p^{n-2}),
\]
for some integer $r$. 
Since $p-1 \geq 2$, we have 
$\sum_{k=1}^{p^{n-1}} (kp)^{p-1} \equiv 0 \pmod{p^{n}}$ and thus 
\[
\sum_{x=1}^{p^n} x^{p-1} \equiv  \sum_{(x,p)=1 \atop  1\leq x \leq p^n } x^{p-1} \pmod{p^n}.
\]

Now consider the endomorphism $\phi: \Z_{p^n}^* \to \Z_{p^n}^*$ defined by $\phi(x) = x^{p-1}$. Its kernel is 
\[\Ker(\phi)= \{ x \in \Z_{p^n}^* \mid x^{p-1}= 1\}.
\]
Since $\Z_{p^n}^*$ is cyclic (as $p$ is odd), the order of the kernel is $p-1$. Given that $|\Z_{p^n}^*|= (p-1)p^{n-1}$, we know the size of  its  image  $|\Ima (\phi)|=p^{n-1}$. Thus, each element in $\Ima (\phi)$ has $p-1$ preimages in $\Z_{p^n}^*$.

Consider the set $A \subseteq \Z_{p^n}^*$ defined as 
\[
A=\{ 1+ sp \mid 0 \leq s< p^{n-1} \},  
\]
whose cardinality  is $p^{n-1}$.  Notice that every element $x \in \Z_{p^n}^*$ satisfies $x^{p-1} \equiv 1 \pmod p$. Hence $\Ima (\phi) \subseteq A$ and by cardinality $\Ima (\phi) = A$. Hence 
\begin{align*}
    \sum_{(x,p)=1 \atop  1\leq x \leq p^n } x^{p-1} &=  (p-1) \sum _{s=0}^{p^{n-1}-1} (1+sp) \\
    &= (p-1) [p^{n-1} + p\sum_{s=0}^{p^{n-1}-1} s] \\
    &=  (p-1) [ p^{n-1} + p \frac{(p^{n-1}-1)p^{n-1}}{2}] \\
 &\equiv -p^{n-1} \pmod{p^n}.    
\end{align*}
This completes the proof of the lemma.

\end{proof}

Our next result states

\begin{lemma}\label{lem:PolynomialZero}
Let $p$ be a prime and $T(x) \in \Z_{p^n}[x]$ be a polynomial of degree at most $p-1$. If $T(x_i) \equiv 0 \pmod{p^n}$ for $p$ integers $x_1, \dots, x_p$ that are all nonequivalent modulo $p$, then $T(x)$ must be the zero polynomial.
\end{lemma}

\begin{proof}
Assume $T(x)= a_0+a_1x+\dots +a_{p-1}x^{p-1}$ with $a_i \in \Z_{p^n}$. The conditions $T(x_i)\equiv 0 \pmod{p^n}$ give the linear  system 
\[
\begin{pmatrix}
1 & x_1 & x_1^2 & \cdots & x_1^{p-1}\\
1 & x_2 & x_2^2 & \cdots & x_2^{p-1}\\
\vdots & \vdots & \vdots & & \vdots\\
1 & x_p & x_p^2 & \cdots & x_p^{p-1}
\end{pmatrix}
\begin{pmatrix}
a_0\\
a_1\\
\vdots\\
a_{p-1}
\end{pmatrix}
=
\begin{pmatrix}
0\\
0\\
\vdots\\
0
\end{pmatrix}
\qquad\text{in } \Z_{p^n}.
\]
The determinant of the coefficient matrix is the Vandermonde determinant
\[
\prod_{1\leq i<j\leq p}(x_j-x_i).
\tag{5}
\]
Since the elements $x_1,\dots,x_p $ are pairwise distinct modulo $p$,
none of the differences $x_j-x_i$ is divisible by $p$. Hence every
$x_j-x_i$ is a unit in $\Z_{p^n}$, and therefore the determinant  is also a unit in $\Z_{p^n}$.
Thus the Vandermonde matrix  above is invertible over $\Z_{p^n}$. It follows
that $a_0=a_1=\cdots=a_{p-1}=0 $ in  $\Z_{p^n}$.
Therefore $T$ is the zero polynomial in $\Z_{p^n}[x]$.
\end{proof}

Now we can prove

\begin{theorem}\label{thm:Cyclic-not-P_{p-1}}
Let $p$ be a prime and $n $ a positive integer. Then the  cyclic $p$-group $G=\Z_{p^n}$ does not satisfy $\P_{p-1}$, and thus does not satisfy $\P_k$ for any $k \geq p-1$.
\end{theorem}

\begin{proof}
For $p=2$ the statement  is supported by the Hall-Paige theorem. For odd primes $p$, consider the automorphisms $\beta_i \in \Aut(G)$ defined by 
\[\beta_i(a)= ia,
\]
for $a \in \Z_{p^n}$ and $i=1, \dots , p-1$.
We assert that no bijection $f: G \to G$ exists such that the maps $a \mapsto ia+f(a)$ are bijections for all $i=1, \dots,p-1$.
Assume, for contradiction, that such a bijection $f: G \to G$ exists, meaning $a \mapsto ia+f(a)$ is a bijection for every $i=0,1,\dots,p-1$.

We define 
\[
T(x)= \sum_{a \in \Z_{p^n}} \big(xa + f(a)\big)^{p-1}  - \sum_{a \in \Z_{p^n}} a^{p-1}.
\]
Then $T(x) $ is a polynomial in $\Z_{p^n}[x]$ of degree $p-1$. 

Given that $ia+f(a)$ covers  all elements of $\Z_{p^n}$, for $a \in \Z_{p^n}$,  it follows that $T(i) = 0 \in \Z_{p^n}$,  for all $i=1, \cdots,p-1$. Also $T(0) = 0$ because $f$ is a bijection of $\Z_{p^n}$. Applying   \autoref{lem:PolynomialZero} we conclude that  $T(x)$ must be the zero polynomial in $\Z_{p^n}[x]$.
But, expanding the binomials in the definition of $T(x)$, we see that the coefficient of $x^{p-1}$ equals  
\[
\sum_{a=1}^{p^n} a^{p-1} \equiv -p^{n-1} \pmod{p^n}. 
\]
As $ -p^{n-1} \not \equiv 0 \pmod{p^n}$ we reach a contradiction and thus proving the proposition.

\end{proof}

The same argument implies that the cyclic obstruction is not limited to cyclic p-groups; it applies whenever the Sylow $p$-subgroup of an abelian group is cyclic, as the next proposition shows.

\begin{proposition}\label{prop: Abelian-P_k}
Assume $G$ is a finite abelian group satisfying $\P_k$.  Then for every prime $p$  with $p \leq k+1$ the Sylow $p$-subgroup $G_p$ is either  trivial or non-cyclic.   
\end{proposition}

\begin{proof}
Assume, for contradiction, that the abelian group $G$ does not satisfy the proposition. Thus, $G$ satisfies $\P_k$ for some integer $k$, and its Sylow $p$-subgroup $G_p$ is cyclic for a prime $p \leq k+1$. This implies $k \geq p-1$, so $G$ satisfies $\P_{p-1}$. Let $G = G_p \times H$, where $H$ is an abelian group of order $m$ with $\gcd(m,p)=1$, and $G_p \cong \Z_{p^n}$ with $n \geq 1$. Consider the automorphisms $B_i \in \Aut(G)$ defined by
\[
B_i (a,h)= (ia, h),  \,  \text{  for all } \, a \in G_p, \,  h \in H,  
\]
for $i=1, \cdots , p-1$.
 By assumption, there exists a bijection $F: G \to G$ such that $B_i + F$ is also a bijection of $G$ for all $i$. We write
 \[
 \pi: G \to \Z_{p^n}
 \]
 for the projection  of $G$  onto $G_p$ composed with the isomorphism $G_p \cong \Z_{p^n}$ (we write $a$ for an element of $G_p$ and its isomorphic element in $\Z_{p^n}$). (So $\pi(B_i(a,h)) = ia$.)
 
 The fact that 
  $B_i + F$ is a bijection  for every $i$, implies that the multiset $\{\pi(B_i(g)+F(g))\mid g \in G\}$  picks every element of $Z_{p^n}$ exactly $m$-times and the same holds for the multiset $\{ \pi(F(g) ) \mid g \in G\}$.

  Similarly with the proof of \autoref{thm:Cyclic-not-P_{p-1}} we define 
 \[
T(x)= \sum_{g \in G } \big( x \pi(g)+ \pi (F(g)) \big)^{p-1}  - m\sum_{a \in \Z_{p^n}} a^{p-1}.
\]
Then $T(x) $ is a polynomial in $\Z_{p^n}[x]$ of degree at most $p-1$. 
Furthermore 
\begin{align}
m \cdot \sum_{a \in \Z_{p^n}} a^{p-1}&= \sum_{g \in G } \big( \pi( B_i(g) +F(g)) \big)^{p-1}=\sum_{g \in G } \big( \pi( B_i(g))  +\pi(F(g)) \big)^{p-1}\\
&= \sum_{h \in H } \sum_{a \in G_p} \big( i a + \pi (F((a,h)) \big)^{p-1}=\sum_{g \in G } \big( i \pi(g)+ \pi (F(g)) \big)^{p-1} .
\end{align}
 Hence $T(i)= 0$ for $i=1, \cdots p-1$. A similar argument shows that $T(0)=0$ because $F$ is a bijection of $G$. The rest of the proof follows th elinbes of that in \autoref{thm:Cyclic-not-P_{p-1}};  $T$ must be the zero polynomial by \autoref{lem:PolynomialZero}, while its leading coefficient is 
 \[
\sum_{g \in G} (\pi(g)) ^{p-1} = m \sum_{a \in \Z_{p^n}}  a^{p-1} \equiv -mp^{n-1} \pmod {p^n}.
 \]
As $(m,p)=1$ the leading coefficient is not zero contradicting the conclusion $T=0$. The proposition now follows.
\end{proof}

We remark here that the requirement that $G$ be abelian in \autoref{prop: Abelian-P_k} cannot be removed. A GAP calculation demonstrates that the non-abelian group $G = \Z_7\rtimes \Z_3$ of order 21 satisfies $\P_2$, even though its Sylow 3-subgroup is cyclic.

The preceding proposition naturally raises the following question,
which may be regarded as a simultaneous analogue of the Hall--Paige
theorem for abelian $p$-groups.

\begin{question}\label{quest:1}
 Let $G$ be a finite abelian $p$-group. Is it true that $G$ satisfies
$\P_{p-1}$ if and only if $G$ is noncyclic? Equivalently, does every
noncyclic finite abelian $p$-group satisfy $\P_{p-1}$?   
\end{question}

Let us remark, that considering the characteristic subgroups of an abelian $p$-group and \autoref{prop: Extension-property}, the question simplifies to groups of type $G \cong \Z_p^{n} \times \Z_{p^{e}}$, where $n, e \geq 1$. The base case $n=e=1$ forms a homocyclic group, for which the question is already confirmed.

\section{Applications} \label{sec:applications}

\subsection{Latin Squares}
Mann in the 1940's introduced complete mappings in connection with construction of orthogonal Latin squares \cite{Man42}. A Latin square of order $n$ is an $n \times n$  array of $n$
symbols in which each symbol occurs exactly once in each row and in each column. 
A {\it transversal}  in a Latin square $L$ of order $n$  is a collection of $n$ cells which do not share any row, column, or symbol.
Two Latin squares $L$ and $M$ of order $n$  are orthogonal  if, given symbols $x$ and $y$ from the alphabets of $L$ and $M$ respectively, there is a unique position $(i,j)$ in the array such that $L(i,j)  = x$ and $M(i,j) = y$,  or equivalently the $n^2$  ordered pairs $(L(i,j) , M(i,j )) $  are distinct. A Latin square orthogonal to a given square $L$ is called an {\it orthogonal mate} of $L$. 
(For further information on Latin squares, their transversals, and orthogonal mates, readers may consult \cite{Wan11} and \cite{Mont24} and the  references therein.)

The Cayley table of finite group $G$  is a natural source of Latin squares (where rows and columns are indexed by the elements of $G$, with the $(g,h)$-entry being  $gh$).
A question asked by Hall and Paige in 1955 was; { \it 
For which groups $G$  does the Cayley table have an orthogonal mate? } The question turns out to be equivalent to determine what groups afford a complete mapping and this in turn yield the famous Hall-Paige conjecture (now theorem). 

In that spirit we try to apply our \autoref{Thm:main} to Latin squares in the following setting; 
Fix a finite group $G$ and for each automorphism $\beta_i \in \Aut(G)$ define a Latin square on $G$ by 
\[
L_i(x,y)=\beta_i(x)y, 
\]
i.e. $L_i$ is the Latin square produced from the Cayley table of $G$ where its rows are permuted by $\beta_i$. Assume now that $G$ satisfies $\P_k$ and let $f$ be a permutation of $G$ so that the maps $x \to \beta_i(x)f(x)$ are also permutations. 
Then the set 
\[
T_f=\{ (x, f(x) ) \mid x \in G\}
\]
forms a transversal for all $L_i$ (because both $f$ and  $\beta_i(x)f(x)$ are permutations on the elements of $G$). In addition, the Latin square 
\[
L_f(x,y):= f(x)^{-1}y 
\]
forms an orthogonal mate of $L_i$, for all $i=1, \dots, k$, as the ordered pairs $( f(x)^{-1}y, \beta_i(x)y)$ can be easily verified to be distinct.  

 Consequently, \autoref{Thm:main} implies \autoref{col:commonTrans+Mate} that we restate here for convenience.
 \begin{corollary*}
     Let $G$ be a finite group whose smallest prime divisor is $p$. If $k < p -1$, then for any automorphisms $\beta_1, \dots, \beta_k$  the Latin squares $L_i(x,y)=\beta_i(x)y$ have a common transversal and a common orthogonal mate.
\end{corollary*}
Observe that if $L_1, \dots, L_k$  happen to be mutually orthogonal, then the corollary implies  that they are not maximal: they can be extended to $k+1$-mutually orthogonal Latin squares. That is,  no mutually orthogonal family of fewer than $p-1$ automorphism-generated Latin squares of a group table is maximal, where $p$ is the smallest prime divisor of $|G|$.

For abelian $p$-groups \autoref{Thm:k<p-1-abelian-p-case} implies 
\begin{corollary}\label{col:Abelian-commonMate}
   If $G$  is a finite  abelian $p$-group and $ k\le p-2$, then every $k$- Latin squares of the form
$L_\beta(x,y)=\beta(x)+y$, with  $\beta\in\Aut(G)$, have a common orthogonal mate of the same form. 
\end{corollary}
We also remark here that if $G$ is noncyclic but  homocyclic then the Corollary above improves from $p-2$ to $p-1$. Hence if $G$ is a noncyclic homocyclic $p$-group, no mutually orthogonal set of $p-1$-automorphism-generated Latin squares is maximal, it can be extended to $p$ such automorphism-generated Latin squares of $G$.

\subsection{Common transversals of subgroups of the same index in a finite abelian group}
Let  $X, Y, Z_i$, for $i =1, 2, \dots, k$, be distinct  subgroups of the same index in  a finite group $G$. As we mentioned in the introduction, the ideas behind the simultaneously complete mappings arose naturally during our effort to generalize the main result in \cite{LS}, i.e., while we were trying to find a common transversal for the subgroups $\{ X, Y, Z_i \mid 2\leq i \leq  k\}$ (the case  $k=1$ was done in \cite{LS}).
The aim for this section is to demonstrate the connection between common transversals of subgroups and simultaneously complete mappings.  The main theorem of the section states:
\begin{theorem}\label{Thm:CommonTrans}
Assume $G$ is a finite abelian group so that $G= X\oplus Y$ with $|X|=|Y|=m$. Assume further that $\{Z_j\}_{j \leq k}$   are subgroups of $G$ of the same order $m$ that intersect trivially both $X$ and $Y$. If $k \leq p-2$, where $p$ is the smallest prime divisor of  $m$, then the  collection of subgroups $\{ X, Y, Z_j \mid j \leq k \}$ share a transversal in $G$.    
\end{theorem}

Before giving its proof (through complete mappings) we should mention that the theorem's premise mandates that all involved subgroups are complemented in $G$, rendering Theorem A in \cite{ALS} directly applicable and its conclusion trivial. Nevertheless, we decided to include its proof here to motivate further research into simultaneously complete mappings.

\begin{proof}
We begin by noting that $G = X \oplus Z_j = Y \oplus Z_j$ for all $j=1, \dots, k$. Hence $X \cong Y$. 
We fix an enumeration $X = \{ x_1, \cdots, x_m\}$ of $X$  and we enumerate the elements of $Y$ with respect of the fixed enumeration of $X$ and the isomorphism $X \cong Y$ (that is,  if $x_i +x_j=x_t$ then $y_i +y_j=y_t$).  
Now we claim that for every $j=1, \dots,  k$,  the elements of the subgroup $Z_j$ can be written as 
\[
Z_j= \{ z_i^j= x_i+ \theta_j(y_i)  \mid i \leq m \}
\]
for some automorphism $\theta_j \in \Aut(Y)$.

To prove the claim, we fix $j$ and we  first observe that every element $z \in Z_j$  has a unique expression $z=x+y$ with $x \in X$ and $y \in Y$  and  distinct elements $z_1, z_2 \in Z_j$ do not share the same $X$ or $Y$-part (as $Z_j \cap X = Z_j \cap Y = \{0\}$).  Hence we can enumerate the elements of $Z_j$ with respect to their $X$-part as 
\begin{equation}\label{eq:z=x+y}
z_i^j =x_i + \theta_j(y_i),
\end{equation}
where $\theta_j$ is a bijection of $Y$. To complete the proof of the claim it remains to show that $\theta_j$ are actually homomorphisms.  But if $z_t^j=z_i^j +z_{i'}^j $,  for some $t$,  then $x_i+x_{i'}=x_t$ and  $ \theta_j(y_i) +\theta_j(y_{i'})= \theta_j(y_t)=\theta_j(y_i+y_{i'}) $, with the last equality following from the way we enumerated $Y$. This proves the claim.

According to \autoref{Thm:main}, there exists a bijection $f:Y \to Y$ so that $f-\theta_j$ are also bijections of $Y$ for all $j=1, \dots, k$, as $k \leq p-2$. Take the set 
\[
F=\{ x_i +f(y_i) \mid i \leq m \}. 
\]
we will show that $F$ is a common transversal for all subgroups  $X, Y, Z_j $, with $1 \leq j \leq k$.  
Clearly $F$ is a transversal of  $X$ (as $f$ is a bijection) and of $Y$. For a fixed $j$, using \eqref{eq:z=x+y} we get  
\[
x_i+f(y_i) \equiv -\theta_j(y_i) +f(y_i) \pmod {Z_j}.
 \]
Because $f-\theta_j$ are bijections for all $j$, all the elements of $F$ are nonequivalent modulo $Z_j$ and thus $F$ is a transversal of $Z_j$.
This completes the proof of the theorem. 

\end{proof}

{\bf Acknowledgment.}
I thank E. Spyridakis for stimulating discussions on the subject. 

\bibliographystyle{amsalpha}
\bibliography{Bibliography}

\end{document}